\documentclass[11pt,letterpaper,reqno]{amsart}
\usepackage{fullpage}
\usepackage{amsmath,amsthm,amsfonts,amssymb,amscd}
\usepackage{empheq}
\usepackage{thmtools}
\usepackage{hyperref}
\usepackage{makecell}
\usepackage{enumerate}
\usepackage{bbm}
\usepackage{stmaryrd}

\allowdisplaybreaks

\let\subsectiontemp\subsection
\renewcommand{\subsection}[1]{ 
    \subsectiontemp{#1} \hfill\vspace{0.5\linespacing} 
} 

\hypersetup{
  colorlinks=true,
  linkcolor=blue,
  citecolor=blue,
  linkbordercolor={0 0 1}
}

\newtheorem{theorem}{Theorem}[section]
\newtheorem{corollary}[theorem]{Corollary}
\newtheorem{lemma}[theorem]{Lemma}

\newtheorem{proposition}[theorem]{Proposition}

\mathtoolsset{showonlyrefs=true}
\numberwithin{equation}{section}
\numberwithin{figure}{section}
\numberwithin{table}{section}

\newcommand{\lrabs}[1]{\left\lvert #1 \right\lvert}
\newcommand{\lrp}[1]{\left(#1\right)}
\newcommand{\lrb}[1]{\left[#1\right]}
\newcommand{\lrbb}[1]{\left\llbracket#1\right\rrbracket}
\newcommand{\lrBigb}[1]{\Big[#1\Big]}
\newcommand{\lrbiggb}[1]{\bigg[#1\bigg]}

\newcommand{\ZZ}{\mathbb{Z}}

\newcommand{\QQ}{\mathbb{Q}}

\newcommand{\Tr}{\mathrm{Tr}}

\newcommand{\eigv}{\textnormal{eigv}}

\newcommand{\one}{\mathbbm{1}}
\newcommand{\nw}{\mathrm{new}}
\newcommand{\Q}{\mathbb{Q}}

\author[E. Ross]{Erick Ross}
\address[E. Ross]{School of Mathematical and Statistical Sciences\\
Clemson University\\
Clemson, SC 29634}
\email{erickjohnross@gmail.com}

\author[A. van Lidth]{Alexandre van Lidth}
\address[A. van Lidth]{Department of Mathematics and Statistics\\
Amherst College\\
Amherst, MA 01002}
\email{alexandrevanlidth@gmail.com}

\author[M. R. Wolf]{Martha Rose Wolf}
\address[M. R. Wolf]{Department of Mathematics\\
University of Michigan\\
Ann Arbor, MI 48109}
\email{martharose.wolf@gmail.com}

\author[H. Xue]{Hui Xue}
\address[H. Xue]{School of Mathematical and Statistical Sciences\\
Clemson University\\
Clemson, SC 29634}
\email{huixue@clemson.edu}


\title{Proportion of Atkin-Lehner sign patterns and Hecke Eigenvalue Equidistribution}

\begin{document}

\begin{abstract}
    Let $N \ge 1$, $k \ge 2$ even, and $\sigma$ denote a sign pattern for $N$. In this paper, we first determine the exact proportion of forms in $S_k(N)$ and $S_k^\nw(N)$ with a given Atkin-Lehner sign pattern $\sigma$. Then we study the asymptotic behavior of the Hecke operators $T_p$ over the subspaces of $S_k(N)$ and $S_k^{\nw}(N)$ with Atkin-Lehner sign pattern $\sigma$.
    In particular, for the $p$-adic Plancherel measure $\mu_p$, we show that the Hecke eigenvalues for $T_p$ over these subspaces are $\mu_p$-equidistributed as $N+k \to \infty$.
\end{abstract}

\maketitle

\section{Introduction}

For $N \ge 1$ and $k \ge 2$ even, let $S_k(N)$ denote the space of cusp forms of weight $k$ and modular group $\Gamma_0(N)$, and let $S_k^\nw(N) \subseteq S_k(N)$ denote its new subspace.
For $p^r \| N$, let $W_{p^r}$ denote the Atkin-Lehner involutions on these spaces, and recall that these $W_{p^r}$ can be extended multiplicatively to $W_Q$ for all $Q \| N$ (i.e. $Q|N$ where $(Q,N/Q)=1$). 

Now, recall that the $W_Q$ involutions diagonalize $S_k(N)$ and $S_k^\nw(N)$ with eigenvalues $\pm1$.
So, we define the Atkin-Lehner sign pattern spaces by specifying these signs for all $p^r \| N$. To be precise, define a sign pattern $\sigma$ for $N$ to be a multiplicative function on the exact divisors of $N$ such that $\sigma(p^r) = \pm 1$ for all $p^r \| N$. Then we have the decomposition of $S_k(N)$  
\begin{align}
    S_k(N) &= \bigoplus_{\sigma} S_k^\sigma(N), \quad S_k^\sigma(N) := \{ f \in S_k(N) : W_{p^r}f = \sigma(p^r) f \ \text{ for each } p^r \| N\}
\end{align}
(and similarly for $S_k^\nw(N)$).
Note that these sign pattern spaces $S_k^\sigma(N)$ are a finer decomposition of the more well-known Fricke sign spaces $S_k^\pm(N)$, which only specify the ``global" sign (i.e. the sign of the Fricke involution $W_N$). 


Now, a natural question one might ask is how all the possible sign patterns are distributed over the spaces $S_k(N)$ and $S_k^\nw(N)$. For both global signs and local sign patterns, we determine the exact proportion of signs / sign patterns appearing.

In the following theorem,  $\one_{\textit{condition}}$ denotes the indicator function for \textit{condition}, and the notation  ``$A(N,k) \sim B(N)$ as $N+k \to \infty$" means that $A(N,k) / B(N) \to 1$ as $N+k \to \infty$. We also take the convention that $0 \sim 0$, because as discussed shortly, both sides of \eqref{eqn:newspace-sgnpatt-proportion} are identically zero in one particular case (see Theorem \ref{thm:W4-sign}).

\begin{theorem} \label{thm:sgnpatt-proportion}
    Consider $N \ge 1$ and $k \ge 2$ even. Then as $N+k \to \infty$, the proportion of cusp forms with global sign $+1$ is given by
    \begin{align}
        \label{eqn:fullspace-globalsgn-proportion}
        \frac{\dim S_k^+(N)}{\dim S_k(N)} &\sim \frac12 + \frac{\one_{N=1}}{2},  \\
        \label{eqn:newspace-globalsgn-proportion}
        \frac{\dim S_k^{\nw,+}(N)}{\dim S_k^\nw(N)} &\sim
            \frac12+\frac{ \one_{N=\textnormal{cubefree square}}}{2} \prod_{p|N} \frac{-1}{p^2-p-1}.
    \end{align}
    Also, consider sign patterns $\sigma$ for $N$. Then as $N+k \to \infty$, the proportion of cusp forms with sign pattern $\sigma$ is given by
    \begin{align}
        \label{eqn:fullspace-sgnpatt-proportion}
        \frac{\dim S_k^\sigma(N)}{\dim S_k(N)} &\sim \frac{1}{2^{\omega(N)}},  \\
        \label{eqn:newspace-sgnpatt-proportion}
        \frac{\dim S_k^{\nw,\sigma}(N)}{\dim S_k^\nw(N)} &\sim 
        \frac{1}{2^{\omega(N)}} \prod_{p^r||N} \lrp{ 1 + \sigma(p^r) \frac{- \one_{r=2}}{p^2 - p -1}}.
    \end{align}
\end{theorem}
We note that this theorem predicts an equal proportion of both signs (and more generally, all sign patterns) in many cases. Over the full space, both signs (respectively, all sign patterns) will have equal proportion in $S_k(N)$ for all $N \ge 2$, by \eqref{eqn:fullspace-globalsgn-proportion} and \eqref{eqn:fullspace-sgnpatt-proportion}. Over the newspace, both global signs (respectively, all sign patterns) will have equal proportion in $S_k^\nw(N)$ for squarefree $N \ge 2$, by \eqref{eqn:newspace-globalsgn-proportion} and \eqref{eqn:newspace-sgnpatt-proportion}. This equiproportion property for the case of newspaces of squarefree level was also shown by Iwaniec-Luo-Sarnak \cite[Equation (2.73)]{ILS} (respectively, Martin \cite[Corollary 3.4]{kimball-2018-refined-dim-formula}). 

The paper \cite{kimball-2023-root-number-bias} recently investigated small biases in global sign, and claimed that each global sign would appear for $50\%$ of newforms as $N+k \to \infty$ \cite[Section 1.3]{kimball-2023-root-number-bias}. However, as one can see from \eqref{eqn:newspace-globalsgn-proportion}, this is not quite accurate; the global sign in $S_k^\nw(N)$ will actually be asymptotically biased towards $(-1)^{\omega(N)}$ whenever $N$ is a cubefree square. One can see this bias numerically from the data calculated in \cite{ross-code}.

We would also like to point out that Martin \cite{kimball-2025-distribution-local-signs} recently investigated small biases in the sign of $W_{p^r}$ for a single fixed prime power $p^r$, and observed a stronger bias at $r=2$. Here, Theorem \ref{thm:sgnpatt-proportion} quantifies this observation by establishing the exact proportion of each sign pattern. 

Now, observe that the proportions \eqref{eqn:fullspace-sgnpatt-proportion} and  \eqref{eqn:newspace-sgnpatt-proportion} are always nonzero, except when $4 \| N$ with $\sigma(4)=+1$ over the newspace $S_k^\nw(N)$. For this reason, we define the notion of ``admissible" sign patterns, where all sign patterns are admissible except for the one case of $4 \| N$ with $\sigma(4)=+1$ over the newspace. 
It is worth noting here that the admissible sign patterns $\sigma$ all have proportion within a factor of $3$ of $1 / 2^{\omega(N)}$. In other words, the admissible sign patterns are not too far off from being equally proportioned. This ``factor of $3$" comes from the lower and upper bounds on the $\prod_{p^r||N} \lrp{ 1 + \sigma(p^r) \frac{- \one_{r=2}}{p^2 - p -1}}$ factor from \eqref{eqn:newspace-sgnpatt-proportion} for admissible $\sigma$: 
\begin{align}
    0.71546 \le \prod_{p \ne 2} \lrp{1-\frac{1}{p^2-p-1}} &\le \prod_{p^r \| N} \lrp{1 + \sigma(p^r) \frac{-\one_{r=2}}{p^2-p-1}} \\
    &\le \prod_p \lrp{1+\frac{1}{p^2-p-1}} \le 2.67412.
    \label{eqn:factor-of-5}
\end{align}

It turns out that for inadmissible sign patterns $\sigma$, there are no newforms at all with sign pattern $\sigma$.
This is because $W_4$ always has Atkin-Lehner sign $-1$ over $S_k^\nw(N)$. This result was originally shown by Atkin and Lehner in \cite[Theorem 7]{atkin-lehner}. We give an alternative proof here based on the trace formula. 
\begin{theorem}\label{thm:W4-sign}
    Let $N \ge 1$ with $4 \| N$. Then $W_4 f = -f$ for every  $f \in S_k^\nw(N)$.
\end{theorem}

Now, recall that the Hecke operators $T_m$ for $(m,N)=1$ commute with the Atkin-Lehner involutions $W_Q$ \cite[Proposition 13.2.6]{CS} (see \cite{CS} for the precise definition of these Hecke operators $T_m$). This means that the Hecke operators preserve the sign pattern spaces $S_k^\sigma(N)$ and $S_k^{\nw,\sigma}(N)$. So for the admissible sign patterns $\sigma$, we can then investigate how the Hecke eigenvalues over $S_k^\sigma(N)$ and $S_k^{\nw,\sigma}(N)$ are distributed. For $m \ge 1$, let $T_m' := m^{-(k-1)/2} T_m$ denote the normalized Hecke operators (over both $S_k^\sigma(N)$ and $S_k^{\nw,\sigma}(N)$). Then for a fixed prime $p$, let $\displaystyle \mu_p(x) :=\frac{p+1}{\pi}\frac{(1-x^2/4)^{1/2}}{(p^{1/2}+p^{-1/2})^2-x^2}dx$ denote the $p$-adic Plancherel measure. We will show that the eigenvalues of $T_p'$ (over both $S_k^\sigma(N)$ and $S_k^{\nw,\sigma}(N)$) are $\mu_p$-equidistributed as $N+k\to\infty$.

\begin{theorem} \label{thm:main-theorem}
    Fix $p$ prime. Consider $N \ge 1$ coprime to $p$, $\sigma$ sign patterns for $N$, and $k \ge 2$ even. Then the Hecke eigenvalues of $T'_p |_{S_k^\sigma(N)}$ are $\mu_p$-equidistributed as $N+k \to \infty$. In other words, for any sequence
    $\{(N_i,k_i,\sigma_i)\}_{i \ge 1}$ 
    with $N_i + k_i \to \infty$, the collection of Hecke eigenvalues $\eigv(T'_m |_{S_{k_i}^{\sigma_i}(N_i)})$ is $\mu_p$-equidistributed as $i \to \infty$.
    Moreover, if we assume it is not the case that $4 \| N$ with $\sigma(4)=+1$, then the Hecke eigenvalues of $T'_p |_{S_k^{\nw,\sigma}(N)}$ are also $\mu_p$-equidistributed as $N+k \to \infty$.
\end{theorem}

Note here that similar $\mu_p$-equidistribution results have been shown before in different settings. For example, Conrey-Duke-Farmer showed $\mu_p$-equidistribution over $S_k(1)$ \cite{CDF} and Serre showed $\mu_p$-equidistribution over $S_k(N)$, $S_k^\nw(N)$, and $S_k(N,\chi)$ \cite{Serre}. More recently, the first author showed $\mu_p$-equidistribution over $S_k^\nw(N,\chi)$ for admissible characters $\chi$ \cite{ross-equid}. Theorem \ref{thm:main-theorem} can be interpreted as strengthening Serre's result over $S_k(N)$ and $S_k^\nw(N)$, showing that $\mu_p$-equidistribution still holds over the much finer sign pattern decompositions $S_k^\sigma(N)$ and $S_k^{\nw,\sigma}(N)$. In fact, this sign pattern decomposition is, in some sense, the finest (natural) decomposition possible; the $S_k^{\nw,\sigma}(N)$ seem to be the smallest subspaces of $S_k^{\nw}(N)$ that still respect the Hecke operators and Galois conjugation \cite[Section 5]{chow-ghitza}.

Lastly, in Section \ref{sec:applications}, we discuss several applications of Theorem \ref{thm:main-theorem}, similar to those in \cite[\S 6]{Serre}. The $\mu_p$-equidistribution shown in Theorem \ref{thm:main-theorem} can be used to obtain information about the Galois orbits of Hecke eigenvalues, or equivalently, about the coefficient fields $\QQ(f)$ of newforms $f \in S_k^{\nw,\sigma}(N)$. In particular, we show that the degrees of these coefficient fields are unbounded as $N \to \infty$. This can be interpreted in terms of a weak form of the generalized Maeda conjecture. We are also able to obtain information about the simple $\QQ$-factors of the Jacobian of the modular curve $X_0(N)$. See Section \ref{sec:applications} for precise details.

We now give an overview of the arguments used in this paper. In Section \ref{sec:trace-Tm-WQ-estimate}, we use a trace formula from Skoruppa-Zagier \cite{Skoruppa1988} to estimate $\Tr\,T_m' {\circ} W_Q$ over both $S_k(N)$ and $S_k^\nw(N)$. Then in Section \ref{sec:trace-over-sgnpatt}, we use the results of Section \ref{sec:trace-Tm-WQ-estimate} to obtain a trace formula for $T_m'$ over the sign pattern spaces. 
Next, in Section \ref{sec:sgnpatt-proportion}, we prove Theorems \ref{thm:sgnpatt-proportion} and \ref{thm:W4-sign}, computing the exact proportion of how often each sign pattern can appear. 
Then, in Section \ref{sec:Hecke-equid}, we use the $\Tr \,T_m'$ estimates of Section \ref{sec:trace-over-sgnpatt} to prove the $\mu_p$-equidistribution of the Hecke eigenvalues over $S_k^\sigma(N)$ and $S_k^{\nw,\sigma}(N)$. In this section, we use the same general strategy used by Serre \cite{Serre}. Finally, in Section \ref{sec:applications}, we discuss applications of Theorem \ref{thm:main-theorem}.

\section{Trace Formula for \texorpdfstring{$T'_m {\circ} W_Q$}{T'\_m o W\_Q}} \label{sec:trace-Tm-WQ-estimate}

For integers $k \geq 2$, let $p_k(t, m)$ denote the Lucas sequence of the first kind; i.e. $p_k(t, m)$ is the $x^{k-2}$ coefficient in the power series expansion of $(mx^2 - tx + 1)^{-1}$. Note that in particular, we have by Clayton et al. \cite[Lemmas 2.2, 2.3]{CLAYTON2024186} that
\begin{align}
    \begin{dcases}
        p_k(t,m) = (k-1)\,m^{(k - 2)/2}
        &\text{if } t^2 - 4m = 0,   \\
        p_k(t,m) \le \frac{2m^{(k - 1)/2}}{\sqrt{\lrabs{t^2-4m}}} 
        &\text{if } t^2-4m < 0. 
    \end{dcases}
    \label{eqn:p-k-Delta}
\end{align}

Following Skoruppa-Zagier \cite{Skoruppa1988}, we also define the following class numbers.
For integers $N \geq 1$ and $\Delta \leq 0$, define $H_N(\Delta)$ via 
\begin{align} \label{eqn:H-N-def}
    H_N(\Delta) = \begin{cases}
        a^2b \left( \cfrac{\Delta / a^2b^2}{N/a^2b} \right) H(|\Delta / a^2b^2|) & \text{if } a^2b^2 \mid \Delta, \\
        0 & \text{otherwise},
    \end{cases}
\end{align}
where $a^2b:=(N,\Delta)$ with $b$ squarefree, $\left( \frac{\cdot}{\cdot} \right)$ denotes the Kronecker symbol, and $H(\cdot)$ denotes the Hurwitz-Kronecker class number. Note that in the case of $\Delta = 0$, $\lrp{\frac{0}{1}}=1$ and $H(0) = \frac{-1}{12}$, so that $H_N(0) = \frac{-1}{12}N$

In the following, we use the notation $Q \| N$ to mean $Q | N$ and $(Q, \frac{N}{Q}) = 1$ and $\ell = \square$ to mean the integer $\ell$ is a square.
In addition, $\sigma_0(N) = \sum\limits_{N' | N} 1$, $\sigma_1(N) = \sum\limits_{N' | N} N'$, and $B(N)$ is defined to be the greatest integer $\ell$ such that $\ell^2 | N$.

\begin{proposition}[Skoruppa-Zagier \cite{Skoruppa1988}, correction by Assaf \cite{Assaf}] \label{prop:S-Z-formula}
  For $N \ge 1$ coprime to $m$, $Q \| N$, and $k \geq 2$ even, we have
  \begin{align*}
      \Tr_{S_k(N)}  \, T'_m {\circ} W_{Q}  &= \sum_{\substack{N' \mid N \\ N/N' \mathrm{ squarefree}}} \mu \left( \frac{Q}{(Q, N')} \right) s'_{k,N'} (m, (Q, N')), \\
      \Tr_{S_k^{\nw}(N)}  \, T'_m {\circ} W_{Q} &= \sum_{N' \mid N} \alpha\left( \frac{N}{N'} \right) s'_{k,N'} (m, (Q, N')).
  \end{align*}
  Here $\mu$ denotes the M\"obius function, $\alpha(N)$ is the nonzero multiplicative function defined on prime powers via
  \begin{align*}
      \alpha(p^r) = \begin{cases}
          -1 & r = 1 \text{ or }  2 \\
          1 &  r = 3 \\
          0 & r \geq 4,
      \end{cases}
  \end{align*}
  and
  \begin{align}
      & s'_{k, N}(m,Q) \\ 
      =
      & -\frac{1}{2 m^{(k-1)/2}} \sum_{Q' \mid Q} \sum_{s} p_{k} \left( \frac{s}{\sqrt{Q'}}, m \right) H_{\frac{N}{Q}} (s^2 - 4m Q') \label{eqn:s-k-N-term1}\\
      & - \frac{1}{2 m^{(k-1)/2}} \sum_{m' \mid m} \min \left( m', \frac{m}{m'} \right)^{k - 1} \left(B(Q), m' + \frac{m}{m'}\right) \left( B\left( \frac{N}{Q} \right), m' - \frac{m}{m'}\right) \label{eqn:s-k-N-term2}\\
      & + \frac{1}{m^{(k-1)/2}} \one_{\substack{k = 2 \\ \frac{N}{Q} = \square}} \sigma_0(Q) \sigma_1(m), \label{eqn:s-k-N-term3}
  \end{align}
  with the summation over $s$ ranging over all $s \in \ZZ$ such that $s^2 \leq 4mQ'$, $Q' | s$, and $\left( \big( \frac{s}{Q'} \big)^2, \frac{Q}{Q'} \right)$ is squarefree.
\end{proposition}

We then estimate the $s'_{k,N}(m,Q)$ terms from this proposition.
In the following, and throughout the entire paper, we use big-$O$ notation with respect to both $N$ and $k$. Note that $m$ here is a fixed constant (so the big-$O$ implied constants will depend on $m$). And $Q$ is considered to be an arbitrary exact divisor of $N$ (so the big-$O$ bounds will only be stated in terms of $N$ and $k$, uniform over arbitrary $Q \parallel N$). 
\begin{lemma} \label{lem:s-k-N-estimates}
    Fix $m \geq 1$.
    Let $N\ge1$ be coprime to $m$, $Q \| N$, and $k \geq 2$ be even. Then 
    \begin{align*}
        s'_{k,N}(m,Q) =   \frac{\one_{m=\square}}{\sqrt{m}} \frac{k - 1}{12}  \frac{N}{Q}  + O(N^{1/2}\sigma_0(N)\log N),
    \end{align*}
    where $s'_{k,N}(m,Q)$ is as defined in Proposition \ref{prop:S-Z-formula}.
\end{lemma}

\begin{proof}
    Note that $\sigma_0(Q) \leq \sigma_0(N) = O(N^\varepsilon)$ and $\sigma_1(m) = O(1)$, so the third term \eqref{eqn:s-k-N-term3} of $s'_{k,N}(m,Q)$ is $O(N^\varepsilon)$.
    Also note that since
    \begin{align*}
        \left(B(Q), m' + \frac{m}{m'}\right) \le 2m,\qquad
        \left( B\left( \frac{N}{Q} \right), m' - \frac{m}{m'}\right) \leq \sqrt{\frac{N}{Q}} \le \sqrt{N},
    \end{align*}
    the second term \eqref{eqn:s-k-N-term2} of $s'_{k,N}(m,Q)$ is of magnitude
    \begin{align*}
        &\leq \frac{1}{2m^{(k-1)/2}} \sum_{m' | m} \min \left( m', \frac{m}{m'} \right)^{k - 1}  (2m)\sqrt{N} \leq \frac{(2m)\sqrt{N}}{{2}} \sum_{m' | m} 1 = O(\sqrt N).
    \end{align*}
    
    Lastly, we consider the first term \eqref{eqn:s-k-N-term1} of $s'_{k,N}(m,Q)$. In the double summation, denote $\Delta := s^2 - 4mQ'$, and we break into two cases: when $\Delta < 0$, and when $\Delta = 0$.
    
    When $\Delta < 0$, we have by \eqref{eqn:p-k-Delta} that
    \begin{align*}
        \left|p_k \left( \frac{s}{\sqrt{Q'}}, m \right) \right| & \leq \frac{2m^{(k-1)/2}}{\sqrt{\left| \frac{s^2}{Q'} - 4m \right|}} = \frac{2m^{(k-1)/2}}{\sqrt{\lrabs{\Delta} / Q'}}.
    \end{align*}
    Using \cite[Lemma 2.2]{GOT}, we can bound the Hurwitz-Kronecker class number by
    \begin{align*}
        H(D) \leq \frac{\sqrt{D} (\log D +2)}{\pi}.
    \end{align*}
    So letting $a^2b = \left( \frac{N}{Q}, \Delta \right)$ we have
    \begin{align*}
        \left| H_{\frac{N}{Q}}(\Delta) \right| & \leq \left| a^2b \left( \cfrac{\Delta / a^2b^2}{N/a^2b} \right) H(|\Delta / a^2b^2|) \right| \\
        & \leq \frac{a^2 b \sqrt{|\Delta / a^2b^2|}  \left( \log\left| \frac{\Delta}{a^2b^2} \right| + 2\right)}{\pi} \\
        & \leq \frac{a \sqrt{|\Delta|} (\log |\Delta| + 2)}{\pi}.
    \end{align*}
    Combining these,
    \begin{align*}
        \left| p_{k} \left( \frac{s}{\sqrt{Q'}}, m \right) H_{\frac{N}{Q}} (\Delta) \right| 
        & \leq \frac{2m^{(k-1)/2}}{\sqrt{\Delta / Q'}} \frac{a \sqrt{|\Delta|} (\log|\Delta| + 2)}{\pi} \\
        & = \frac{2m^{(k-1)/2} a \sqrt{Q'} (\log|\Delta| + 2)}{\pi} \\
        & \leq \frac{2m^{(k-1)/2} \sqrt{N} (\log(4mN) + 2)}{\pi} \\
        & = m^{(k-1)/2} O(N^{1/2} \log N),
    \end{align*}
    where the last inequality comes from the facts that $a^2 \leq \frac{N}{Q} \leq \frac{N}{Q'}$ and $|\Delta| \le 4mQ' \le 4mN$.
    
    Now, note that in \eqref{eqn:s-k-N-term1}, the number of terms in the summation over $Q'|Q$ is equal to $\sigma_0(Q) \le \sigma_0(N)$.
    Furthermore, in \eqref{eqn:s-k-N-term1}, the number of terms in the summation over $s$ is $\leq 4\sqrt{m/Q'} = O(1)$ since $Q' | s$ and $|s| \leq 2\sqrt{mQ'}$.
    Thus, the contribution of the $\Delta < 0$ terms in $s'_{k,N}(m,Q)$ is $O(N^{1/2} \sigma_0(N) \log N)$.
    
    Finally, we consider the $\Delta = 0$ terms, which will make up the main term of $s'_{k,N}(m,Q)$.
    Recall that $(m,Q')=1$. Then note that $\Delta := s^2 - 4mQ' =0$ implies that $m=\square$ and $Q'=\square$. Additionally, $(Q')^2 \,|\, s^2=4mQ'$ implies that $Q' \,|\, 4$. Hence we must have $Q' = 1 \text{ or }4$. It is straightforward to verify that the first case occurs precisely when $m=\square$ and $4 \nmid Q$ (for $Q'=1$, $s=\pm 2 \sqrt m$); and the second case occurs precisely when $m=\square$ and $4 | Q$ (for $Q'=4$, $s=\pm 4 \sqrt m$).
    This means that in all cases, the number of $\Delta = 0$ terms appearing in $s'_{k,N}(m,Q)$ is exactly $2\,\one_{m=\square}$.
     
    And when $\Delta = 0$, we have by \eqref{eqn:p-k-Delta} and \eqref{eqn:H-N-def} that
    \begin{align}
        p_k\lrp{\frac{s}{\sqrt{Q'}}, m} &= (k-1)\,m^{(k-2)/2} \\
        \text{and}\quad H_{\frac{N}{Q}} (s^2 - 4m Q') & = H_{\frac{N}{Q}} (0) = \frac{-1}{12} \frac{N}{Q}.
    \end{align}
    
    Thus, extracting the $\Delta=0$ terms from $s'_{k, N}(m,Q)$, we obtain
    \begin{align*}
        &-\frac{1}{2m^{(k-1)/2}} \sum_{Q' | Q} \sum_{\substack{s^2 = 4mQ' \\ Q' | s}} p_{k} \left( \frac{s}{\sqrt{Q'}}, m \right) H_{\frac{N}{Q}} (s^2 - 4m Q') \\
        =& -\frac{1}{2m^{(k-1)/2}} 
        \Big(2 \one_{m=\square}\Big) 
        \Big( (k-1) m^{(k-2)/2} \Big)
        \Big( \frac{-1}{12} \frac{N}{Q} \Big)  \\
        = &  \frac{\one_{m=\square}}{\sqrt{m}} \frac{k - 1}{12} \frac{N}{Q},
    \end{align*}
    completing the proof.
\end{proof}

With this estimate of $s'_{k,N}(m,Q)$, we can then estimate $\Tr \, T'_m {\circ} W_Q$.
In the following, we use the notation $\lrbb{A}_{\textit{condition}}$ to denote that the term $A$ only appears when $\textit{condition}$ is satisfied.

\begin{proposition}\label{prop:tr-estimate}
    Fix $m \geq 1$.
    Let $N\ge1$ be coprime to $m$, $Q \| N$, and $k \geq 2$ be even.
    Then
    \begin{align*}
        \Tr_{S_k(N)} \, T'_m {\circ} W_Q &= \frac{\one_{m = \square}}{\sqrt{m}}  \frac{k - 1}{12} \psi(N) \one_{Q = 1} + O(N^{1/2} \sigma_0(N)^2 \log N), \\
        \Tr_{S_k^{\nw}(N)} \, T'_m {\circ} W_Q &=   \frac{\one_{m = \square}}{\sqrt{m}}  \frac{k - 1}{12} \psi^{\nw}(N) \eta(Q) 
         + O(N^{1/2} \sigma_0(N)^2 \log N),
    \end{align*}
    where $\psi$, $\psi^\nw$, and $\eta$ are the nonzero multiplicative functions defined on prime powers via
    \begin{align*}
        \psi(p^r) &:= p^r \lrp{1+\frac 1p}, \\
        \psi^{\nw}(p^r) &:= p^r \lrp{  
            1 - \frac 1p - \lrbb{ \frac{1}{p^2} }_{r\ge 2} + 
            \lrbb{ \frac{1}{p^3} 
            }_{r\ge 3} 
        }, \\
         \eta(p^r) &:= \frac{-\one_{r=2}}{p^2 - p - 1}.
    \end{align*}
    Note in particular that when $N$ is squarefree, $\eta(Q) = \one_{Q = 1}$.
\end{proposition}
\begin{proof}
    Define the multiplicative function $\theta_M(d) := \frac{d}{(M, d)}$. Observe that $\mu {\circ} \theta_M$ is also multiplicative since $\theta_M(d_1), \theta_M(d_2)$ are coprime whenever $d_1,d_2$ are coprime.
 
    To show the first desired identity, it suffices (by Lemma \ref{lem:s-k-N-estimates} and Proposition \ref{prop:S-Z-formula}) to show that
    \begin{align} \label{eqn:sum-mu}
      \sum_{\substack{N' \mid N \\ N/N' \mathrm{ squarefree}}} \mu \lrp{ \frac{Q}{(Q, N')} } \frac{N'}{(Q, N')}  = \psi(N)\one_{Q = 1}.
  \end{align}

  Observe that $\one_{N/N' \ \text{squarefree}} = \mu^2(N/N')$, $\mu \lrp{ \frac{Q}{(Q, N')} } = \mu \lrp{ \frac{N/N'}{(N/N', N/Q)} } = \mu {\circ} \theta_{N/Q}(N/N')$, and $\frac{N'}{(Q, N')} = \theta_Q(N')$. This means that 
  \begin{align}
      \text{LHS of \eqref{eqn:sum-mu}} &= [[\mu^2 \cdot\mu {\circ}\theta_{N/Q}] * \theta_Q](N) \\
      & = \prod_{p^r \| N} [[\mu^2 \cdot \mu {\circ}\theta_{N/Q}] * \theta_Q](p^r) \\
      & = \prod_{p^r \| N} \sum_{i = 0}^r   \mu^2(p^i)\, \mu {\circ} \theta_{N/Q}(p^i) \, \theta_Q(p^{r - i}) \\
      & = \prod_{p^r \| N} \Big( \theta_Q(p^r) + \mu {\circ} \theta_{N/Q}(p) \,\theta_Q(p^{r - 1}) \Big) \\
      &= \prod_{p^r \| N}  
      \begin{cases}
        p^r + p^{r-1} & \text{if } p \nmid Q \\
         0 & \text{if } p \mid Q
      \end{cases} \\
      &= \psi(N)\one_{Q = 1}.
  \end{align}
  Note that here, the second-to-last equality comes from the facts that
    \begin{align}
        \theta_Q(p^s) &= \begin{cases}
            p^s & \text{if } p \nmid Q \\
            1 & \text{if } p | Q
        \end{cases} \qquad \text{and}\qquad
        \mu{\circ}\theta_{N/Q}(p) = \begin{cases}
            1 & \text{if } p \nmid Q \\
            -1 & \text{if } p | Q.
        \end{cases} 
    \end{align}

  Similarly, to show the second desired identity, it suffices (by Lemma \ref{lem:s-k-N-estimates} and Proposition \ref{prop:S-Z-formula}) to show that
  \begin{align} \label{eqn:sum-alpha}
      \sum_{N' | N} \alpha\left( \frac{N}{N'} \right) \frac{N'}{(Q, N')}  = \psi^\nw(N) \eta(Q).
  \end{align}
  And here,
  \begin{align}
      \text{LHS of \eqref{eqn:sum-alpha}} &= [\alpha * \theta_Q](N) \\
      & = \prod_{p^r \| N} [\alpha * \theta_Q](p^r) \\
      & = \prod_{p^r \| N} \sum_{i = 0}^r \alpha(p^{i}) \theta_Q(p^{r-i}) \label{eqn:sum-alpha-conv} \\
      &= \prod_{p^r \parallel N} \Big( \theta_Q(p^r) - \theta_Q(p^{r-1}) - \lrbb{\theta_Q(p^{r-2})}_{r \ge 2}
      + \lrbb{\theta_Q(p^{r-3})}_{r \ge 3} \Big) \\
      &= \prod_{p^r \parallel N} \begin{cases}
         \psi^\nw(p^r) & \text{if } p \nmid Q \\
         - \one_{r=2} & \text{if } p \mid Q
      \end{cases} \\
      &= \psi^\nw(N) \prod_{p \mid Q} \frac{- \one_{v_p(N)=2} }{ \psi^\nw(p^2) } \\
      &= \psi^\nw(N) \prod_{p^r \| Q} \frac{- \one_{r=2} }{ p^2-p-1 } \\
      &= \psi^\nw(N) \eta(Q),
  \end{align}
  as desired.
\end{proof}

\section{Trace Formula for \texorpdfstring{$T'_m$}{T'\_m} over the sign pattern spaces} \label{sec:trace-over-sgnpatt}

In \cite[Proposition 3.2]{kimball-2018-refined-dim-formula}, Martin derived a formula for $\dim S_k^{\nw,\sigma}(N)$ in terms of the traces $\Tr_{S_k^{\nw}(N)}W_Q$. Here, we prove a similar formula for $\Tr_{S_k^{\sigma}(N)}T'_m$ and $\Tr_{S_k^{\nw,\sigma}(N)}T'_m$  in terms of the traces $\Tr_{S_k(N)} T'_m {\circ} W_Q$ and $\Tr_{S_k^\nw(N)} T'_m {\circ} W_Q$, respectively.

We note that the identity \eqref{eqn:sgnpatts-orthogonal} underlying this lemma can also be interpreted as a special case of orthogonality for characters over a finite group. (The set of all exact divisors of $N$ forms a group via the binary operation $Q_1 * Q_2 := \frac{Q_1 Q_2}{(Q_1,Q_2)^2}$. And sign patterns for $N$ make up the characters on this group.)

\begin{lemma}\label{lem:martin-identity} For $N\geq1$,
\begin{align}
    \Tr_{S_k^{\sigma}(N)}T'_m &= \frac{1}{2^{\omega(N)}}\sum_{Q\mid\mid N}\sigma(Q)\Tr_{S_k(N)} \, T'_m {\circ} W_Q,  \\ 
    \Tr_{S_k^{\nw,\sigma}(N)}T'_m &= \frac{1}{2^{\omega(N)}}\sum_{Q\mid\mid N}\sigma(Q)\Tr_{S_k^\nw(N)} \, T'_m {\circ} W_Q.  
\end{align}
\end{lemma}
\begin{proof}
We show this result just for the full space $S_k(N)$, and the argument for the newspace $S_k^\nw(N)$ is identical.

First, observe that 
\begin{align}
    \sum_{Q\mid\mid N}\sigma(Q)\sigma '(Q) 
    &= \prod_{p^r \| N} \sum_{Q \| p^r} \sigma(Q) \sigma'(Q) =  \prod_{p^r \| N} \lrp{ 1 +  \sigma(p^r) \sigma'(p^r) }\\
    &= \prod_{p^r \| N} 2 \one_{\sigma(p^r) = \sigma'(p^r)} = 2^{\omega(N)}\one_{\sigma =\sigma '}. \label{eqn:sgnpatts-orthogonal}
\end{align}

Then, proceeding by splitting $\Tr_{S_k(N)}T'_m{\circ}W_Q$ into a sum of $\Tr_{S_k^{\sigma '}(N)}T'_m{\circ}W_Q$ for all sign patterns $\sigma'$ of $N$:
    \begin{align*}
        \frac{1}{2^{\omega(N)}}\sum_{Q\mid\mid N}\sigma(Q)\Tr_{S_k(N)} \, T'_m{\circ}W_Q 
        &= \frac{1}{2^{\omega(N)}}\sum_{Q \mid\mid N}\sigma(Q)\sum_{\sigma '} \Tr_{S_k^{\sigma '}(N)}T'_m {\circ}W_Q \\
        &= \frac{1}{2^{\omega(N)}}\sum_{Q \mid\mid N}\sum_{\sigma '}\sigma(Q)\sigma '(Q)\Tr_{S_k^{\sigma '}(N)}T'_m 
        \\&= \frac{1}{2^{\omega(N)}}\sum_{\sigma'} 2^{\omega(N)} \one_{\sigma = \sigma'}\Tr_{S_k^{\sigma'}(N)}T_m'
        \\&= \Tr_{S_k^{\sigma}(N)}T'_m, 
    \end{align*}
    as desired.
\end{proof}

We can then estimate $\Tr_{S_k^\sigma(N)} T'_m$ and $\Tr_{S_k^{\nw,\sigma}(N)} T'_m$ by applying Proposition \ref{prop:tr-estimate} to Lemma \ref{lem:martin-identity}.
\begin{corollary} \label{cor:tr-tm-sgnpatt-estimate}
    Fix $m \geq 1$.
    Then for $N \ge 1$ coprime to $m$ and $k \geq 2$ even,
\begin{align}
        \label{eqn:trace-tm-fs-estimate}
        \Tr_{S_k^\sigma(N)} \, T'_m &= \frac{\one_{m = \square}}{\sqrt{m}}  \frac{k - 1}{12} \frac{\psi(N)}{2^{\omega(N)}}  + O(N^{1/2}\sigma_0(N)^2\log N), \\
        \label{eqn:trace-tm-ns-estimate}
        \Tr_{S_k^{\nw,\sigma}(N)} \, T'_m &=   \frac{\one_{m = \square}}{\sqrt{m}}  \frac{k - 1}{12} \frac{\psi^{\nw}(N)}{2^{\omega(N)}} \prod_{p^r||N} \lrp{ 1 + \sigma(p^r) \frac{- \one_{r=2}}{p^2 - p -1}} 
         + O(N^{1/2} \sigma_0(N)^2 \log N).
         \qquad 
    \end{align}
\end{corollary}
\begin{proof}
Using Lemma \ref{lem:martin-identity} and then Proposition \ref{prop:tr-estimate}, we obtain
\begin{align*}
    \Tr_{S_k^\sigma(N)}T_{m}'
    &=\frac{1}{2^{\omega(N)}}\sum_{Q\mid\mid N}\sigma(Q)\Tr_{S_k(N)}T_m'{\circ} W_Q
    \\
    &=\frac{1}{2^{\omega(N)}} \sum_{Q \| N} \sigma(Q) \frac{\one_{m=\square}}{\sqrt{m}}\frac{k-1}{12}\psi(N)  \one_{Q=1} +O(N^{1/2}\sigma_0(N)^2\log N)\\
    &=\frac{\one_{m=\square}}{\sqrt{m}}\frac{k-1}{12}\frac{\psi(N)}{2^{\omega(N)}}+O(N^{1/2}\sigma_0(N)^2\log N),\\
    \Tr_{S_k^{\nw,\sigma}(N)}T_{m}'&=\sum_{Q\mid\mid N}\sigma(Q)\Tr_{S_k^{\nw}(N)}T_m'{\circ} W_Q
    \\&=\frac{1}{2^{\omega(N)}} \sum_{Q\|N}  
    \sigma(Q)
    \frac{\one_{m=\square}}{\sqrt{m}}\frac{k-1}{12}\psi^{\nw}(N)
    \eta(Q)+O(N^{1/2}\sigma_0(N)^2\log N)
    \\
    &=\frac{\one_{m=\square}}{\sqrt{m}}\frac{k-1}{12}\frac{\psi^{\nw}(N)}{2^{\omega(N)}}\prod_{p^r||N} \lrp{ 1 + \sigma(p^r) \frac{- \one_{r=2}}{p^2 - p -1}}+O(N^{1/2}\sigma_0(N)^2\log N),
\end{align*}
as desired.
\end{proof}

We point out here that the main term of \eqref{eqn:trace-tm-fs-estimate} always grows faster than the error term 
\\ 
$O(N^{1/2} \sigma_0(N)^2\log N) = O(N^{1/2+\varepsilon})$. Furthermore, the main term of \eqref{eqn:trace-tm-ns-estimate} grows faster than the error term, assuming it is not the case that $4 \| N$ with $\sigma(4)=+1$ (because after excluding this one particular case, the $\prod_{p^r||N} \lrp{ 1 + \sigma(p^r) \frac{- \one_{r=2}}{p^2 - p -1}}$ factor in \eqref{eqn:trace-tm-ns-estimate} is $\ge 1/2$ by \eqref{eqn:factor-of-5}). In other words, the main terms of this corollary grow faster than the corresponding error terms for all admissible sign patterns $\sigma$.

\section{Proportion of Sign Patterns}
\label{sec:sgnpatt-proportion}

In this section, we show Theorems \ref{thm:sgnpatt-proportion} and \ref{thm:W4-sign}.

{ 
\renewcommand{\thetheorem}{\ref{thm:sgnpatt-proportion}}
\begin{theorem}
    Consider $N \ge 1$ and $k \ge 2$ even. Then as $N+k \to \infty$, the proportion of cusp forms with global sign $+1$ is given by
    \begin{align}
        \frac{\dim S_k^+(N)}{\dim S_k(N)} &\sim \frac12 + \frac{\one_{N=1}}{2},  \\
        \frac{\dim S_k^{\nw,+}(N)}{\dim S_k^\nw(N)} &\sim
            \frac12+\frac{ \one_{N=\textnormal{cubefree square}}}{2} \prod_{p|N} \frac{-1}{p^2-p-1}.
    \end{align}
    Moreover, consider sign patterns $\sigma$ for $N$. Then as $N+k \to \infty$, the proportion of cusp forms with sign pattern $\sigma$ is given by
    \begin{align}
        \label{eqn:proportion-fullspace}
        \frac{\dim S_k^\sigma(N)}{\dim S_k(N)} &\sim \frac{1}{2^{\omega(N)}},  \\
        \label{eqn:proportion-newspace}
        \frac{\dim S_k^{\nw,\sigma}(N)}{\dim S_k^\nw(N)} &\sim 
        \frac{1}{2^{\omega(N)}} \prod_{p^r||N} \lrp{ 1 + \sigma(p^r) \frac{- \one_{r=2}}{p^2 - p -1}}.
    \end{align}
\end{theorem}
\addtocounter{theorem}{-1}
}
\begin{proof}

First, substitute $m=1$ and $Q=1$ into Proposition \ref{prop:tr-estimate} to obtain the dimensions of the entire spaces:
\begin{align} 
    \label{eqn:entire-fullspace-dim}
    \dim S_k(N) &= \frac{k-1}{12} \psi(N) + O(N^{1/2+\varepsilon}), \\
    \label{eqn:entire-newspace-dim}
    \dim S_k^\nw(N) &= \frac{k-1}{12} \psi^\nw(N) + O(N^{1/2+\varepsilon}).
\end{align}

Now, for the first two claims, we use Proposition \ref{prop:tr-estimate} at $m=1$ and $Q=1,N$ to obtain
\begin{align}
    \dim S_k^+(N) 
    &= \frac12\lrp{ \Tr_{S_k(N)}W_1+\Tr_{S_k(N)}W_N} \\
    &= \frac{k-1}{12} \psi(N)  \frac{ 1 + \one_{N=1} }{2} + O(N^{1/2+\varepsilon}), \\ 
    \dim S_k^{\nw,+}(N) &= \frac12\lrp{ \Tr_{S_k^\nw(N)}W_1+\Tr_{S_k^\nw(N)}W_N} \\
    &= \frac{k-1}{12} \psi^\nw(N)  \frac{ 1 + \eta(N) }{2} + O(N^{1/2+\varepsilon}) \\
    &= \frac{k-1}{12} \psi^\nw(N) \frac{1+\one_{N = \text{cubefree square}}  \prod_{p|N} \frac{-1}{p^2-p-1}
    }{2} + O(N^{1/2+\varepsilon}).
\end{align}
Comparing these estimates with \eqref{eqn:entire-fullspace-dim} and \eqref{eqn:entire-newspace-dim} then yields the desired result.

For the last two claims, we substitute $m=1$ into Corollary \ref{cor:tr-tm-sgnpatt-estimate} to obtain
\begin{align}
    \label{eqn:tempSksigma-dim}
    \dim S_k^\sigma(N)&=\frac{1}{2^{\omega(N)}}\frac{k-1}{12}\psi(N)+O(N^{1/2+\varepsilon}),
    \\
    \dim S_k^{\nw,\sigma}(N)&=\frac{1}{2^{\omega(N)}}\frac{k-1}{12}\psi^{\nw}(N)\prod_{p^r||N}\lrp{1+\sigma(p^r)\frac{-\one_{r=2}}{p^2-p-1}}+O(N^{1/2+\varepsilon}).
\end{align}
Comparing these estimates with \eqref{eqn:entire-fullspace-dim} and \eqref{eqn:entire-newspace-dim} similarly yields the desired result. 
\end{proof}

Recall that we defined the notion of admissible sign patterns, where all sign patterns are admissible except for the case of $4 \| N$ with $\sigma(4) = +1$ over the newspace. We introduced this definition because the proportions \eqref{eqn:proportion-fullspace} and \eqref{eqn:proportion-newspace} are always nonzero except in this one exceptional case.

In fact, it turns out that $W_4$ always has Atkin-Lehner sign $-1$ over $S_k^\nw(N)$. This means that for inadmissible sign patterns $\sigma$, there are no forms at all in $S_k^\nw(N)$ with sign pattern $\sigma$. 

{ 
\renewcommand{\thetheorem}{\ref{thm:W4-sign}}
\begin{theorem}
    Let $N \ge 1$ with $4 \| N$. Then $W_4 f = -f$ for every $f \in S_k^\nw(N)$.
\end{theorem}
\addtocounter{theorem}{-1}
}
\begin{proof}
    It suffices to show that $\Tr_{S_k^\nw(N)} W_1 + \Tr_{S_k^\nw(N)} W_4 = 0$.
    Using Proposition \ref{prop:S-Z-formula}, we compute
    \begin{align}
        & \Tr_{S_k^\nw(N)} W_1 + \Tr_{S_k^\nw(N)} W_4 \\
        & = \sum_{N' | N} \alpha\left( \frac{N}{N'} \right) 
        \lrbiggb{ s'_{k, N'}(1, (1, N')) + s'_{k, N'}(1,(4, N')) } \\
        & = \sum_{L | \frac{N}{4}} \sum_{N' \in \{L, 2L, 4L\}} \alpha\left( \frac{N}{N'} \right) 
        \lrbiggb{ s'_{k, N'}(1, (1, N')) + s'_{k, N'}(1,(4, N')) } \\
        & = \sum_{L | \frac{N}{4}} 
        \bigg[ 
        \alpha\left( \frac{N}{L} \right) \Big[
            s'_{k, L}(1,1) + s'_{k, L}(1,1)
        \Big]
        +
        \alpha\left( \frac{N}{2L} \right) \Big[
            s'_{k, 2L}(1,1) + s'_{k, 2L} (1,2)
        \Big] \\
        & \quad \quad \quad \quad + \alpha \left( \frac{N}{4L} \right) 
        \Big[
            s'_{k, 4L}(1,1) + s'_{k, 4L}(1,4)
        \Big]
        \bigg] \\
        & = \sum_{L | \frac{N}{4}} \alpha \left( \frac{N}{4L} \right) 
        \lrbiggb{
        -2 s'_{k, L}(1,1) - s'_{k, 2L}(1,1) - s'_{k, 2L}(1,2) + s'_{k, 4L}(1,1) + s_{k, 4L}(1,4)
        }\\ 
        &=: \sum_{L | \frac{N}{4}} \alpha \left( \frac{N}{4L} \right) 
        \lrbiggb{
        *
        }. 
    \end{align}
    Now, we show that $[*] = 0$.
    Computing each of the terms appearing in $[*]$ (see \eqref{eqn:s-k-N-term1}-\eqref{eqn:s-k-N-term3} for the definition of $s'_{k,L}$) yields
    \begin{align*}
        s'_{k,L}(1,1) & = -\frac{1}{2} \lrb{p_k(0,1)H_L(-4) + 2p_k(1,1)H_L(-3) + 2p_k(2,1)H_L(0) + B(L) - 2\one_{\substack{k = 2 \\ L = \square}}}, \\
        s'_{k, 2L}(1,1) & = -\frac{1}{2} \lrBigb{p_k(0,1)H_{2L}(-4) + 2p_k(1,1)H_{2L}(-3) + 2p_k(2,1)H_{2L}(0) + B(L)}, \\
        s'_{k, 2L}(1,2) & = -\frac{1}{2} \bigg[ p_k(0,1)H_L(-4) + 2p_k(1,1) H_L(-3) + 2p_k(2,1) H_L(0) + p_k(0,1) H_L(-8) \\
        & \quad \quad \quad \quad + 2p_k(\sqrt{2}, 1) H_L(-4) + B(L) - 4 \one_{\substack{k = 2 \\ L = \square}} \bigg], \\
        s'_{k, 4L}(1,1) & = -\frac{1}{2} \lrb{p_k(0,1)H_{4L}(-4) + 2p_k(1,1) H_{4L}(-3) + 2p_k(2,1) H_{4L}(0) + 2B(L) - 2\one_{\substack{k = 2 \\ L = \square}}}, \\
        s'_{k, 4L}(1,4) & = -\frac{1}{2} \bigg[2p_k(1,1) H_L(-3) + p_k(0,1) H_L(-8) + 2p_k(\sqrt{2}, 1) H_L(-4) + p_k(0,1)H_L(-16) \\
        & \quad \quad \quad \quad  + 2p_k(2,1)H_L(0) + 2B(L) - 6 \one_{\substack{k = 2 \\ L = \square}} \bigg].
    \end{align*}
    Then grouping all the terms of $[*]$ by the $p_k(t,m)$ yields
    \begin{align}
        [*] =&  \; \frac{1}{2}p_k(0,1) 
        \lrBigb{3H_L(-4) + H_{2L}(-4) - H_{4L}(-4) - H_L(-16)} \\
        & + p_k(1,1) \lrBigb{ 2H_L(-3) + H_{2L}(-3) - H_{4L}(-3) } \\
        &  + p_k(2,1) \lrBigb{ 2H_L(0) + H_{2L}(0) - H_{4L}(0) }  \\
        = & \; \frac{1}{2}p_k(0,1) \left[ \frac{3}{2} \left( \frac{-4}{L} \right) - \frac{3}{2} \left( \frac{-16}{L} \right) \right] \\
        & + p_k(1,1) \frac{\one_{3 {\nmid} L}}{3} \left[ 2 \left( \frac{-3}{L} \right) + \left( \frac{-3}{2L} \right) - \left( \frac{-3}{4L} \right) \right] \\
        & + p_k(2,1) \frac{-1}{12} \lrBigb{ 2L + 2L - 4L } \\
        = & \; 0,
    \end{align}
    as desired.
\end{proof}

\section{\texorpdfstring{$\mu_p$}{mu\_p}-equidistribution of Hecke eigenvalues}
\label{sec:Hecke-equid}

In this section, we prove Theorem \ref{thm:main-theorem}. 
For a fixed prime $p$, let $\mu_p$ denote the $p$-adic Plancherel measure $\displaystyle \mu_p(x) =\frac{p+1}{\pi}\frac{(1-x^2/4)^{1/2}}{(p^{1/2}+p^{-1/2})^2-x^2}dx$. 
Then over the spaces $S_k^\sigma(N)$ and $S_k^{\nw,\sigma}(N)$ (for admissible sign patterns $\sigma$), 
we show that the Hecke eigenvalues 
for $T_p'$ are $\mu_p$-equidistributed as $N+k \to \infty$. 

{ 
\renewcommand{\thetheorem}{\ref{thm:main-theorem}}
\begin{theorem}
    Fix $p$ prime. Consider $N \ge 1$ coprime to $p$, $\sigma$ sign patterns for $N$, and $k \ge 2$ even. Then the Hecke eigenvalues of $T'_p |_{S_k^\sigma(N)}$ are $\mu_p$-equidistributed as $N+k \to \infty$. In other words, for any sequence
    $\{(N_i,k_i,\sigma_i)\}_{i \ge 1}$ 
    with $N_i + k_i \to \infty$, the collection of Hecke eigenvalues $\eigv(T'_m |_{S_{k_i}^{\sigma_i}(N_i)})$ is $\mu_p$-equidistributed as $i \to \infty$.
    Moreover, if we assume it is not the case that $4 \| N$ with $\sigma(4)=+1$, then the Hecke eigenvalues of $T'_p |_{S_k^{\nw,\sigma}(N)}$ are also $\mu_p$-equidistributed as $N+k \to \infty$.
\end{theorem}
\addtocounter{theorem}{-1}
}
\begin{proof}
We just state the argument here for $S_k^{\nw,\sigma}(N)$, and an identical argument works for the full space $S_k^{\sigma}(N)$.

To show asymptotic $\mu_p$-equidistribution of the Hecke eigenvalues, it suffices (e.g. by Kuipers-Niederreiter \cite[Theorem 1.1]{KN}, Serre \cite[Proposition 2]{Serre}) to show that 
\begin{equation} \label{eqn:equid-cheby} 
\frac{1}{\dim S_k^{\nw,\sigma}(N)}\sum_{\lambda \in \text{eigv}\lrp{T'_p|_{S_k^{\nw,\sigma}(N)}}} X_n(\lambda) 
\longrightarrow
\int_{-2}^{2} X_n(x) \mu_p(x) \quad\text{as } N+k\to\infty
\end{equation}
for the Chebyshev polynomials $X_n(x) := U_n(x/2)$ for $n \ge 0$, since they span the space of all polynomials, which is a dense subspace of all continuous functions on $[-2,2]$. 

The integral on the right hand side of \eqref{eqn:equid-cheby} has value $p^{-n/2} \one_{\text{$n$ even}}$, by Serre \cite[Equation (20)]{Serre}. 

Additionally, it is well-known (see \cite[Lemme 1]{Serre}, for example) that
$$T_{p^n}' = X_n(T_p').$$
This means that
\begin{align}
    \text{LHS of \eqref{eqn:equid-cheby}} &= 
    \frac{1}{\dim S_k^{\nw,\sigma}(N)}\sum_{\lambda \in \text{eigv}\lrp{T'_p|_{S_k^{\nw,\sigma}(N)}}} X_n(\lambda) \\
    &= \frac{
        \Tr_{S_k^{\nw,\sigma}(N)} X_n(T'_{p})
    }{
        \Tr_{S_k^{\nw,\sigma}(N)} T'_{1}
    } \\
    &= \frac{
        \Tr_{S_k^{\nw,\sigma}(N)} T'_{p^n}
    }{
        \Tr_{S_k^{\nw,\sigma}(N)} T'_{1}
    } \\
    &\to p^{-n/2} \one_{\text{$n$ even}} \quad \text{as } N+k \to \infty,
\end{align}
 where for the last step, we used Corollary \ref{cor:tr-tm-sgnpatt-estimate} at $m=p^n$ and $m=1$ (since $\sigma$ is an admissible sign pattern). This verifies \eqref{eqn:equid-cheby}, completing the proof.
\end{proof}

\section{Applications}\label{sec:applications}

Here, we state some applications of Theorem \ref{thm:main-theorem} analogous to those in Serre \cite[\S6]{Serre}.

For each $N\geq1$ and $\sigma$ admissible sign patterns for $N$, choose an eigenbasis $\{f_i\}_{i=1}^{\dim S_k^\sigma(N)}$ of $S_k^\sigma(N)$ with respect to the Hecke operators $T_m$ satisfying $(m,N)=1$. 
For each $m$ coprime to $N$, let $x_{i,m}$ denote the $T_m$ eigenvalue of $f_i$.
Then for each $i$, let $K_i^\sigma := \Q(\{x_{i,m}\}_{(m,N) = 1})$ denote the totally real field generated by the $x_{i,m}$ (see Cohen-Str\"omberg \cite[Corollary 10.3.7(c), Remark 10.6.3(b)]{CS}).
Then let $s(N,k,\sigma)_r$ be the number of $i$ such that $[K_i^\sigma:\mathbb{Q}]=r$ (which is well-defined as the $K_i^\sigma$ are independent of the eigenbasis chosen, up to permutation (cf. Serre, \S6.1 \cite{Serre})).

Similarly for $S_k^{\nw, \sigma}(N)$, we have the newform basis $\{f^\nw_i\}_{i = 1}^{\dim S_k^{\nw, \sigma}(N)}$, and we define $K_i^{\nw, \sigma}$ and $s(N,k,\sigma)^\nw_r$ analogously.
Then the following improves a result of Serre \cite[Th\'eor\`eme 5]{Serre} (and Binder \cite[Corollary 10.0.5]{Binder}), by extending to the sign pattern spaces.

\begin{corollary} \label{cor:theorem-5}
    Fix $k\geq 2$ even, $p$ prime, and $r \geq 1$. Then for $N$ coprime to $p$ and $\sigma$ admissible sign patterns for $N$,
    \begin{align*}
        \frac{s(N, k, \sigma)_r}{\dim S_k^\sigma(N)} &\to 0  \qquad \text{as } N \to \infty,\\
        \frac{s(N, k, \sigma)^\nw_r}{\dim S_k^{\nw, \sigma}(N)} &\to 0 \qquad \text{as } N \to \infty.
    \end{align*}
\end{corollary}

\begin{proof}
    We show the proof only over the full space, as the proof over the newspace is identical. For $N$ coprime to $p$, let $s(N,k,\sigma,p)_r$ denote the number of indices $i$ such that $[\mathbb{Q}(x_{i,p}):\mathbb{Q}] \leq r$. Note that $s(N,k,\sigma,p)_r \geq s(N,k,\sigma)_r$, so it suffices to show $$\frac{s(N,k,\sigma,p)_r}{\dim S_k^\sigma(N)}\to 0$$ as $N\to \infty$ while $(N,p)=1$.
    Note that if $x=x_{i,p}$ satisfies $[\mathbb{Q}(x_{i,p}):\mathbb{Q}]\leq r,$ then $x$ also satisfies:
    \begin{enumerate}
        \item $x$ is a totally real algebraic integer of degree $\leq r$;
        \item For $\tau \in \text{Gal}(\overline{\QQ}/\QQ)$, the Galois conjugate $x^\tau$ of $x$ satisfies $|x^\tau|\leq 2p^{(k-1)/2}$. 
    \end{enumerate}
    Let $A(p,k,r)$ be the set of all $x$ satisfying conditions (1) and (2). Such a set must be finite, as the characteristic polynomials of its elements $x$ must have degree $\leq r$ and bounded integer coefficients. Let $A'\subseteq [-2,2]$ be the dilation of $A(p,k,r)$ under the mapping $x\mapsto \frac{x}{p^{(k-1)/2}}.$ The set $A'$ is finite, so it has $\mu_p$ measure $0$, and it is independent of $N$ and $\sigma$. Then since the $\frac{x_{i,p}}{p^{(k-1)/2}}$ are $\mu_p$-equidistributed as $N\to\infty$ by Theorem \ref{thm:main-theorem}, the proportion of $i$ such that $\frac{x_{i,p}}{p^{(k-1)/2}}\in A'$ (and hence $x_{i,p}\in A$) tends to $0$ as $N\to\infty$ (e.g. by Serre \cite[Proposition 1]{Serre}). This is equivalent to saying that $\frac{s(N,k,\sigma,p)_r}{\dim S_k^\sigma(N)}\to0$ as $N\to\infty,$ completing the proof.
\end{proof}

We note here that over the sign pattern newspaces $S_k^{\nw,\sigma}(N)$, we have $K_i^\sigma = \QQ(f_i)$, where $\{f_i\}_{i = 1}^{\dim S_k^{\nw,\sigma}(N)}$ is the newform basis for $\dim S_k^{\nw,\sigma}(N)$. 
From this perspective, Corollary \ref{cor:theorem-5} can be interpreted as a weak form of the Maeda Conjecture for general level.

Recall that the classical Maeda Conjecture states that $\deg \QQ(f) = \dim S_k(1)$ (i.e. is maximum possible) for each normalized Hecke eigenform $f \in S_k(1)$. For squarefree levels $N$, a corresponding conjecture has also been stated for the newspaces. Since the Hecke operators restrict to the $S_k^{\nw,\sigma}(N)$; for newforms $f \in S_k^\nw(N)$, the maximum possible degree of $\QQ(f)$ is $\dim S_k^{\nw,\sigma_f}(N)$, where $\sigma_f$ denotes the sign pattern for $f$. In particular, Tsaknias and Chow-Ghitza posited a version of the Maeda Conjecture for squarefree level: that for newforms $f \in S_k^\nw(N)$, $\deg \QQ(f) = \dim S_k^{\nw,\sigma_f}(N)$ for sufficiently large weight $k$ (see \cite[Conjecture 2.4]{Tsaknias} and \cite[Conjecture 5.2]{chow-ghitza}). 

For general level $N$, however, the picture is not as clear. In this case, there seem to be plenty of examples where $\deg \QQ(f)$ is not maximal. However, weaker forms of the Maeda Conjecture have still been investigated. For example, Roberts conjectured that for any fixed weight $k \ge 6$, $\deg \QQ(f) = 1$ for only finitely many newforms $f$ \cite[Conjecture 1.1]{Roberts2018}. Later, Martin made the same conjecture for arbitrary fixed degree $r$ \cite[Question 3]{MartinConjecture}. Here, Corollary \ref{cor:theorem-5} can be interpreted as a density $1$ version of these conjectures. In particular, for any fixed $r \ge 1$, Corollary \ref{cor:theorem-5} shows that the proportion of newforms $f \in S_k^{\nw,\sigma}(N)$ with $\deg \QQ(f) \le r$ tends to $0$ as $N \to \infty$ (along $N$ coprime to $p$).

We can also use Corollary \ref{cor:theorem-5} to obtain information about the \textit{largest} degree of $K_i^\sigma$ and $K_i^{\nw,\sigma}$. The following result follows immediately from Corollary \ref{cor:theorem-5}.

\begin{corollary}[cf. Serre, Th\'eor\`eme 6 \cite{Serre}] \label{cor:theorem-6}
    Fix $k \geq 2$ even and $p$ prime.
    Let $r(N,k, \sigma)$ be the maximum degree among the $K^\sigma_i$ associated with $(N,k,\sigma)$, where $\sigma$ is an admissible sign pattern for $N$, and similarly define $r(N, k, \sigma)^\nw$.
    Then for $N$ coprime to $p$ and $\sigma$ admissible sign patterns for $N$,
    \begin{align*}
        r(N,k, \sigma) &\to \infty  \qquad \text{as } N \to \infty, \\
        r(N,k, \sigma)^\nw &\to \infty  \qquad \text{as } N \to \infty.
    \end{align*}
\end{corollary}

 Now, let $J_0(N)$ denote the Jacobian of the modular curve $X_0(N)$. Note that the Atkin-Lehner involutions act naturally on $X_0(N)$, and thus also act on $J_0(N)$ by functoriality. Up to $\Q$-isogeny, we have the decomposition:
 \begin{align}
     J_0(N)\cong \prod_{\sigma} J^{\sigma}_0(N),
 \end{align}
 where $W_Q=\sigma(Q)$ on $J^{\sigma}_0(N)$ for all $Q||N$. Now, each $J_0^{\sigma}(N)$ can be decomposed into a product of $\Q$-simple factors:
 \begin{align}
     J_0^{\sigma}(N)\cong \prod A^\sigma_j.
 \end{align}
 
 Following the same argument as Serre (and Ribet \cite{Ribet-twists}), we see that the number of factors $A^\sigma_j$ with dimension $r$ is exactly $\frac{1}{r}s(N,2,\sigma)_r$ for each $r\ge1$. Now, Corollary \ref{cor:theorem-6} implies the following.

\begin{corollary}[cf. Serre, Th\'eor\`eme 7 \cite{Serre}]
    Fix $p$ prime.
    Let $r(N,\sigma)$ be the largest dimension of the simple $\Q$-factors of $J_0^\sigma(N)$, where $\sigma$ is an admissible sign pattern for $N$. 
    Then for $N$ coprime to $p$ and $\sigma$ admissible sign patterns for $N$,
    \begin{align*}
        r(N, \sigma) \to \infty  \qquad \text{as } N \to \infty.
    \end{align*}
\end{corollary}

\section*{Acknowledgements}
We would like to thank Kimball Martin for his helpful comments on this paper. We would also like to thank the referee for many insightful comments, and especially for pointing out the reference \cite{atkin-lehner} to Theorem \ref{thm:W4-sign}. 

This research was supported by NSF grant DMS-2349174. Research of Hui Xue is supported by  Simons Foundation Grant MPS-TSM-00007911.

\bibliographystyle{plain}
\bibliography{bib}

\end{document}